\documentclass[oneside,12pt]{amsart}
\usepackage{eurosym}
\usepackage{amsfonts}
\usepackage{amsmath,amssymb,amsthm,color}
\usepackage{amsopn}
\usepackage{latexsym}
\usepackage{float}
\usepackage{bm}
\usepackage[utf8]{inputenc}

\newcommand{\mn}{\mathfrak n }

\newcommand{\mm}{\mathfrak m }

\newcommand{\mz}{\mathfrak z }

\newcommand{\mv}{\mathfrak v }
\newcommand{\mh}{\mathfrak h }
\newcommand{\ma}{\mathfrak a }
\newcommand{\mgg}{\mathfrak g }

\newcommand{\so}{\mathfrak{so} }

\newcommand{\bil}{g}
\newcommand{\lela}{ g(}
\newcommand{\rira}{)}
\newcommand{\lra}{\longrightarrow}

\newcommand{\R}{\mathbb R}

\newcommand{\C}{\mathbb C}

\DeclareMathOperator{\ad}{ad}

\numberwithin{equation}{section}

 \newtheorem{teo}{Theorem}[section]
 \newtheorem{pro}[teo]{Proposition}
 \newtheorem{cor}[teo]{Corollary}
 \newtheorem{lm}[teo]{Lemma}
 \newtheorem{defi}[teo]{Definition}

 \theoremstyle{definition}
 \newtheorem{ex}[teo]{Example}
 \newtheorem{remark}[teo]{Remark}

\newcommand{\nc}{\newcommand}
\nc{\Iso}{\operatorname{Iso}}
 \nc{\iso}{\mathfrak{iso}}
 \nc{\sso}{\mathfrak{so}}
\nc{\Ad}{\operatorname{Ad}} 
\nc{\Sym}{\mathrm{Sym}}
 \nc{\pr}{\operatorname{pr}} 
 \nc{\Dera}{\operatorname{Dera}} \nc{\Auto}{\operatorname{Auto}}
 
\begin{document}

\title{Killing forms on $2$-step nilmanifolds}
\author{Viviana del Barco}
\address{Universidad Nacional de Rosario, CONICET, Rosario, Argentina}
\thanks{V. del Barco partially supported by FAPESP
grant 2015/23896-5.}
\email{delbarc@fceia.unr.edu.ar}

\author{Andrei Moroianu}
\address{Laboratoire de Math\'ematiques d'Orsay, Univ. Paris-Sud, CNRS, Universit\'e Paris-Saclay, 91405 Orsay, France }
\email{andrei.moroianu@math.cnrs.fr}

\begin{abstract} We study left-invariant Killing $k$-forms on simply connected $2$-step nilpotent Lie groups endowed with a left-invariant Riemannian metric. 
For $k=2,3$, we show that every left-invariant Killing $k$-form is a sum of Killing forms on the factors of the de Rham decomposition. Moreover, on each irreducible factor,  non-zero Killing $2$-forms define (after some modification) a bi-invariant orthogonal complex structure and non-zero Killing $3$-forms arise only if the Riemannian Lie group is naturally reductive when viewed as a homogeneous space under the action of its isometry group. In both cases, $k=2$ or $k=3$, we show that the space of left-invariant Killing $k$-forms of an irreducible Riemannian 2-step nilpotent Lie group is at most one-dimensional.
\end{abstract}

\maketitle

\section{Introduction}

Killing forms on Riemannian manifolds are exterior forms whose covariant derivative with respect to the Levi-Civita connection is totally skew-symmetric. They have been introduced by physicists as first integrals of the equation of motion \cite{pw}, and were intensively studied in the mathematical literature, especially in the framework of special holonomy on compact manifolds \cite{bms}, \cite{au}, \cite{s}. A recent striking application of Killing forms was the classification of compact manifolds carrying two conformally equivalent non-homothetic metrics with special holonomy \cite{a}. 

Examples of manifolds with non-parallel Killing $k$-forms are however quite rare, except for $k=1$, where Killing $1$-forms are just metric duals of Killing vector fields. For $k=2$, non-trivial Killing $2$-forms exist on nearly Kähler manifolds (the fundamental form), on the round spheres
$\mathbb{S}^p$ (eigenforms corresponding to the least eigenvalue of the Laplacian on co-closed $2$-forms) and on some ambitoric and ambikähler manifolds in dimension $4$, cf. \cite{gm}. For $k=3$, non-trivial Killing $3$-forms can be constructed on nearly Kähler $6$-dimensional manifolds, on the round spheres, on Sasakian and $3$-Sasakian manifolds. Further examples are given by the fundamental $3$-form on every naturally reductive homogeneous space, and by the torsion of metric connections with parallel skew-symmetric torsion \cite{cms}. On the contrary, it was recently shown that compact Riemannian manifolds with negative scalar curvature do not carry any non-vanishing Killing $3$-form \cite{lau}.

A usual source of examples in Riemannian geometry is provided by simply connected Lie groups endowed with left-invariant metrics. On such manifolds, the study of left-invariant objects reduces to an algebraic equation on the level of the Lie algebra, which can be tackled under additional assumptions. In this paper we are going to investigate left-invariant Killing $k$-forms on $2$-step nilpotent Lie groups.

Our main result is the description of left-invariant Killing $2$- and $3$-forms on simply connected $2$-step nilpotent Lie groups endowed with left-invariant Riemannian metrics. We first show that every such form is a sum of left-invariant Killing forms on the factors in the de Rham decomposition. This reduces the problem to the study of de Rham irreducible Riemannian Lie groups $(N,g)$. 

Assuming that $(N,g)$ is irreducible, we show that Killing $2$-forms are obtained from bi-invariant orthogonal complex structures (see Definition \ref{def:bii}); this was also shown recently by A. Andrada and I. Dotti \cite{AD19}. In addition, using a result of C. Gordon \cite{GO2}, we prove that non-zero Killing $3$-forms exist if and only if $(N,g)$ is a naturally reductive homogeneous space under the action of its isometry group. Moreover, still under the irreducibility hypothesis, we show that the spaces of left-invariant Killing $2$-forms and Killing $3$-forms are at most one dimensional. 

As a consequence of the above results, we obtain the low-dimensional classification of simply connected 2-step nilpotent Lie groups carrying a left-invariant Riemannian metric with non-zero Killing $2$-forms (up to dimension $8$) or Killing $3$-forms (up to dimension $6$).

{\sc Acknowledgment.} We would like to thank the anonymous referee for several pertinent remarks and suggestions, which led to the low-dimensional classification results mentioned above, and to the global improvement of the presentation.

\section{Preliminaries: Riemannian geometry of $2$-step nilpotent Lie groups}
In this section we describe the geometry of $2$-step nilpotent Lie groups endowed with left-invariant Riemannian metrics by means of their Lie algebras and the inner product that is induced on them.

Let $N$ be a connected Lie group endowed with a left-invariant Riemannian metric $g$, and let $\mn$ denote the Lie algebra of $N$. Left translations by elements $a$ of the Lie group, $L_a:N\lra N$ with $h\mapsto L_a(h)=ah$, are isometries and the metric $g$ is determined by its value on the tangent space at the identity, which we identify with $\mn$. Let $\nabla$ denote the Levi-Civita connection of $(N,g)$. Koszul's formula evaluated on left-invariant vector fields $X,Y,Z$ on $N$ reads
\begin{equation}\label{eq.Koszul}
\lela \nabla_X Y,Z\rira=\frac12\{ \lela [X,Y],Z\rira+\lela [Z,X],Y\rira+\lela [Z,Y],X\rira\}.
\end{equation}

This formula shows in particular that the covariant derivative of a left-invariant vector field with respect to another left-invariant vector field is again left-invariant. From now on we will identify a left-invariant vector field $X$ with its value $x\in\mn$ at the origin, so that \eqref{eq.Koszul} becomes 
\begin{equation}\label{eq.Koszul1}
\nabla_xy=\frac12\left( [x,y]-\ad_x^*y-\ad_y^*x\right),
\end{equation}
where $\ad_x^*$ denotes the adjoint of $\ad_x$ with respect to the corresponding inner product $g$ on the Lie algebra of $N$.

The center and the commutator of a Lie algebra $\mn$ are, respectively,
\[\mz=\{z\in\mn\ |\   [x,z]=0, \,\mbox{for all }x\in\mn\},\qquad
\mn'=[\mn,\mn]:=\mathrm{span}\{ [x,y]\ |\   x,y\in\mn\}.
\]
A Lie algebra $\mn$ is said to be $2$-step nilpotent if it is non abelian and $\ad_x^2=0$ for all $x\in\mn$. 
Equivalently, $\mn$ is $2$-step nilpotent if $0\neq \mn'\subseteq \mz$.

\begin{remark} \label{class} Let $\mathcal N_p$ denote the set of isomorphism classes of real $2$-step nilpotent Lie algebras. It is well known that $\mathcal N_1=\mathcal N_2=\emptyset$, $\mathcal N_3=\{\mh_3\}$ (the real Heisenberg Lie algebra), $\mathcal N_4=\{\R\oplus \mh_3\}$, ${\rm card }(\mathcal N_5)=3$ and ${\rm card }(\mathcal N_6)=7$ (cf. \cite{MA} for the explicit list).
\end{remark}
 
For the rest of the section we assume that $\mn$ is $2$-step nilpotent and $N$ is its corresponding simply connected $2$-step nilpotent Lie group. We shall describe the main geometric properties of $(N,g)$ through linear objects in the metric Lie algebra $(\mn,g)$, following the work of Eberlein \cite{EB}.

Let $\mv$ be the orthogonal complement of $\mz$ in $\mn$ so that $\mn=\mv\oplus\mz$ as an orthogonal direct sum of vector spaces.  Each central element $z\in\mz$ defines an endomorphism $j(z):\mv\lra\mv$ by the equation
\begin{equation}\label{eq:jota}
\lela j(z)x,y\rira =\lela z,[x,y]\rira, \quad \mbox{ for all } x,y\in\mv.
\end{equation}
It is straightforward that $j(z)$ belongs to $\sso(\mv)$, the Lie algebra of skew-symmetric endomorphisms of $\mv$ with respect to $\bil$.

The linear map $j: \mz \to \sso(\mv)$ captures important geometric information of the Riemannian manifold $(N,\bil)$. In particular, using \eqref{eq.Koszul1} we readily obtain that the covariant derivative of left-invariant vector fields can be expressed as follows
\begin{equation}\label{eq:nabla}\left\{
\begin{array}{ll}
\nabla_x y=\frac12 \,[x,y] & \mbox{ if } x,y\in\mv,\\
\nabla_x z=\nabla_zx=-\frac12 j(z)x & \mbox{ if } x\in\mv,\,z\in\mz,\\
\nabla_z z'=0& \mbox{ if } z, z'\in\mz.
\end{array}\right.
\end{equation}

The linear map $j:\mz\lra \so(\mv)$ is injective if and only if the commutator of $\mn$ coincides with its center, that is, $\mn'=\mz$. 

\begin{remark}\label{an}
Let $\ma$ denote the kernel of $j:\mz\to\so(\mv)$ and let $\ma^\perp$ be its orthogonal in $\mz$. Then $\ma$ is an abelian ideal of $\mn$ and $\mn_0:=\mv\oplus\ma^\bot$ is a
$2$-step nilpotent ideal of $\mn$, whose commutator coincides with its center: $\mn_0'=\ma^\bot$. Moreover, $\mn=\mn_0\oplus \ma$ is an orthogonal direct sum.
\end{remark}

Note that by \eqref{eq:nabla}, the Levi-Civita covariant derivative vanishes whenever one of the arguments belongs to $\ma$, so $\ma$ corresponds to a flat Riemannian factor in the de Rham decomposition of $N$.

The following is a technical result that is included for future use.
\begin{lm} \label{lm:imjzv} Let $\mn=\mv\oplus\mz$ be a $2$-step nilpotent Lie algebra, then $\mv=\sum_{z\in\mz}{\rm Im} j(z)$. 
\end{lm}
\begin{proof} Suppose that some $x\in\mv$ belongs to the kernel of $j(z)$ for every $z\in\mz$. Then by \eqref{eq:jota}, $x$ commutes with every $y\in\mv$, so eventually $x\in\mz\cap\mv=0$. This shows that 
\begin{equation}\label{int}\bigcap_{z\in\mz}\ker j(z)=0.\end{equation}
Moreover $\ker j(z)={\rm Im} j(z)^\bot$ since $j(z)$ is skew-symmetric, thus $$0=\bigcap_{z\in\mz}{\rm Im} j(z)^\bot=(\sum_{z\in\mz}{\rm Im} j(z))^\bot$$ which gives $\mv=\sum_{z\in\mz}{\rm Im} j(z)$. 
\end{proof}

\section{Killing forms on Lie groups}
In this section we introduce the concept of (differential) Killing forms on arbitrary Riemannian manifolds and we present basic properties of the left-invariant ones on Lie groups.\smallskip

\begin{defi}
A Killing $k$-form on a Riemannian manifold $(M,g)$ is a differential $k$-form $\alpha$ that satisfies 
\begin{equation}\label{eq:Killform}
\nabla_Y \alpha=\frac1{k+1}Y\lrcorner\ {\rm d}\alpha
\end{equation}
for every vector field $Y$ in $M$, where $\lrcorner$ denotes the contraction of differential forms by vector fields.	
\end{defi}
Equivalently, $\alpha$ is a Killing form if and only if  $Y\lrcorner \nabla_Y\alpha=0$ for every vector field $Y$ (see \cite{Se03}). If $X$ is a vector field on $M$ and $\alpha$ is its metric dual $1$-form, i.e. $\alpha=\lela X, \cdot\rira$, then $\alpha$ is a Killing $1$-form if and only if $X$ is a Killing vector field. 
\smallskip

Let $N$ be a Lie group with Lie algebra $\mn$ and let $g$ be a left-invariant metric on $N$. A left-invariant differential $k$-form  $\alpha$ on $N$, i.e. such that $L_a^*\alpha=\alpha$ for all $a\in N$, is determined by its value at the identity; hence one can see such forms as elements in $\Lambda^k\mn^*$. The covariant and exterior derivatives preserve left-invariance, thus $\alpha\in \Lambda^k\mn^*$ is a Killing form if and only if 
\begin{equation}\label{eq:Killforminv}
\nabla_y \alpha=\frac1{k+1}\ y\lrcorner\ {\rm d}\alpha, \quad \mbox{ for all } y\in \mn,
\end{equation}
where ${\rm d}:\Lambda^k\mn^*\lra \Lambda^{k+1}\mn^*$ is the Lie algebra differential. 

As mentioned before, \eqref{eq:Killforminv} is  equivalent to $y\lrcorner \nabla_y\alpha=0$ for all $y\in \mn$. We shall now give a different expression of this condition. 

Recall that every skew-symmetric endomorphism $f$ of $\mn\simeq\mn^*$ extends as a derivation (also denoted by $f$) of the exterior algebra $\oplus_{k\ge 0}\Lambda^k\mn^*$ by the formula 
\begin{equation}\label{der} f(\alpha):=\sum_{i=1}^pf(u_i)\wedge u_i\lrcorner\alpha,\qquad \mbox{ for all }\alpha\in \Lambda^k\mn^*,
\end{equation}
where $\{u_1, \ldots, u_p\}$ is any orthonormal basis of $(\mn,\bil)$. 

Since $\nabla_y$ is a skew-symmetric endomorphism of $\mn$, given $\alpha\in \Lambda^k\mn^*$ and $y\in \mn$, we can write $\nabla_y\alpha=\sum_{i=1}^p\nabla_yu_i\wedge (u_i\lrcorner\ \alpha)$ and compute:
\begin{eqnarray}
y\lrcorner \nabla_y\alpha&=&y\lrcorner \sum_{i=1}^p\nabla_yu_i\wedge (u_i\lrcorner\ \alpha)= \sum_{i=1}^p\left(\lela y,\nabla_y u_i\rira u_i\lrcorner \alpha-\nabla_y u_i\wedge (y\lrcorner \ u_i\lrcorner\ \alpha)\right)\nonumber\\
&=& -\sum_{i=1}^p\lela u_i,\nabla_y y\rira u_i\lrcorner \alpha-\sum_{i=1}^p\nabla_y u_i\wedge (y\lrcorner \ u_i\lrcorner\ \alpha)\label{eq:f1}\\
&=&-\nabla_y y\lrcorner \alpha- \sum_{i=1}^p \nabla_y u_i\wedge (y\lrcorner \ u_i\lrcorner\ \alpha)\nonumber.
\end{eqnarray}
Therefore $\alpha\in \Lambda^k\mn^*$ defines a left-invariant Killing form $(N,g)$ if and only if 
\begin{equation*}
 \nabla_y y\lrcorner \alpha+ \sum_{i=1}^p \nabla_y u_i\wedge (y\lrcorner \ u_i\lrcorner\ \alpha)=0, \quad \mbox{ for all } y\in \mn.
\end{equation*}

We interpret this condition in the particular case of Riemannian $2$-step nilpotent Lie groups. Assume $(N,g)$ is a simply connected $2$-step nilpotent Lie group endowed with a left-invariant metric. In the notation of the preliminaries, the Lie algebra  $\mn$ of $N$  decomposes as a sum of orthogonal subspaces $\mn=\mv\oplus \mz$. Let $\{e_1,\ldots,e_n\}$ and $\{z_1, \ldots,z_m\}$ be orthonormal bases of $\mv$ and $\mz$, respectively.
\begin{pro} \label{pro:p1} A $k$-form $\alpha$ on a $2$-step nilpotent Lie algebra $\mn$ defines a left-invariant Killing form on $(N,g)$ if and only if for every $x\in \mv$ and $z\in \mz$ the following conditions hold:
\begin{eqnarray}
\sum_{i=1}^n [x,e_i]\wedge (x\lrcorner \ e_i\lrcorner\ \alpha) &=& \sum_{t=1}^mj(z_t)x\wedge (x\lrcorner \ z_t\lrcorner\ \alpha),\label{eq:p1}\\
\sum_{i=1}^n j(z)e_i\wedge (z\lrcorner \ e_i\lrcorner\ \alpha) &=&0,\label{eq:p2}\\
 \sum_{i=1}^n [x,e_i]\wedge (z\lrcorner \ e_i\lrcorner\ \alpha) &=&2(j(z)x)\lrcorner \alpha + \sum_{i=1}^n j(z)e_i\wedge (x\lrcorner \ e_i\lrcorner\ \alpha) \nonumber \\
&&\qquad\qquad\qquad+ \sum_{t=1}^mj(z_t)x\wedge (z\lrcorner \ z_t\lrcorner\ \alpha).\label{eq:p3}
\end{eqnarray}
\end{pro}

\begin{proof} Let $y\in \mn$ and write $y=x+z$ with $x\in \mv$ and $ z\in \mz$. For any $\alpha\in \Lambda^k\mn^*$, we can use \eqref{eq:f1} and \eqref{eq:nabla} to describe the $(k-1)$-form 
$y\lrcorner \nabla_y\alpha$ in terms of the orthonormal bases of $\mv$  and $\mz$ above:
\begin{eqnarray}
y\lrcorner \nabla_y \alpha&=&-(j(z)x)\lrcorner \alpha + \frac12 \sum_{i=1}^n [x,e_i]\wedge (x\lrcorner \ e_i\lrcorner\ \alpha)+\frac12 \sum_{i=1}^n [x,e_i]\wedge (z\lrcorner \ e_i\lrcorner\ \alpha)  \label{eq:f5} \\
&& \qquad \qquad
-\frac12 \sum_{i=1}^n j(z)e_i\wedge (x\lrcorner \ e_i\lrcorner\ \alpha) -\frac12 \sum_{i=1}^n j(z)e_i\wedge (z\lrcorner \ e_i\lrcorner\ \alpha) \nonumber\\
&& \qquad  \qquad-\frac12 \sum_{t=1}^m j(z_t)x\wedge (x\lrcorner \ z_t\lrcorner\ \alpha)  -\frac12 \sum_{t=1}^m j(z_t)x\wedge (z\lrcorner \ z_t\lrcorner\ \alpha) . \nonumber
\end{eqnarray}
Using this expression, it is easy to verify that $y\lrcorner \nabla_y\alpha=0$  for all $y\in \mn$ whenever \eqref{eq:p1}--\eqref{eq:p3} hold, hence $\alpha$ is a Killing form in this case.

Suppose now that $\alpha$ is a Killing form and hence $y\lrcorner \nabla_y\alpha=0$ for all $y\in \mn$. In particular, taking $y=x\in \mv$ (i.e. $z=0$) in \eqref{eq:f1} and using $x\lrcorner \nabla_x\alpha=0$ we obtain \eqref{eq:f5}.
 Similarly, putting $y=z\in \mz$ in \eqref{eq:f5} we get \eqref{eq:p2}. Replacing these two conditions in \eqref{eq:f5} for an arbitrary $y\in \mn$, and using the fact that $y\lrcorner \nabla_y\alpha=0$ we obtain \eqref{eq:p3}.
\end{proof}

The decomposition $\mn=\mv\oplus\mz$ induces a decomposition of the space of $k$-forms on $\mn$ 
\begin{equation}\label{eq:lambdad}
\Lambda^k\mn^*=\bigoplus_{l=0}^k \Lambda^l\mv^*\otimes \Lambda^{k-l}\mz^*.
\end{equation}
Accordingly, for $\sigma\in \Lambda^k\mn^*$, we write $\sigma_l$ for the projection of $\sigma$ on $\Lambda^l\mv^*\otimes \Lambda^{k-l}\mz^*$ with respect to this decomposition.

Given a  $k$-form $\alpha$ on $\mn$, the three equations in  Proposition \ref{pro:p1} give equalities between $(k-1)$-forms determined by different contractions of $\alpha$. Projecting these equalities to $\Lambda^l\mv^*\otimes \Lambda^{k-1-l}\mz^*$ for each $l=0, \ldots, k-1$, we get the following result.
\begin{cor} \label{cor:killgen} A $k$-form $\alpha$ on a $2$-step nilpotent Lie algebra $\mn$ defines a left-invariant Killing form on $(N,g)$ if and only if for every $l=0,\ldots,k-1$, $x\in \mv$ and $z\in \mz$ the following conditions hold:
\begin{eqnarray}
\sum_{i=1}^n [x,e_i]\wedge (x\lrcorner \ e_i\lrcorner\ \alpha_{l+2}) &=& \sum_{t=1}^mj(z_t)x\wedge (x\lrcorner \ z_t\lrcorner\ \alpha_l),\label{eq:pp1}\\
 \sum_{i=1}^n j(z)e_i\wedge (z\lrcorner \ e_i\lrcorner\ \alpha_l) &=&0,\label{eq:pp2}\\
\sum_{i=1}^n [x,e_i]\wedge (z\lrcorner \ e_i\lrcorner\ \alpha_{l+1})&=&2( j(z)x)\lrcorner \alpha_{l+1} + \sum_{i=1}^n j(z)e_i\wedge (x\lrcorner \ e_i\lrcorner\ \alpha_{l+1})\label{eq:pp3}\\
&&\qquad+ \sum_{t=1}^m j(z_t)x\wedge (z\lrcorner \ z_t\lrcorner\ \alpha_{l-1}). \nonumber
\end{eqnarray}
\end{cor}
Notice that for $l=0$ the right hand side of \eqref{eq:pp1}, the left hand side of \eqref{eq:pp2}, and the last two terms of \eqref{eq:pp3} vanish, independently of the degree of $\alpha$. 
The same occurs for $l=k-1$ in the left hand side of \eqref{eq:pp1}.

\section{Invariant Killing $2$-forms on $2$-step nilpotent Lie groups}

In this section we apply Corollary \ref{cor:killgen} to show that Killing $2$-forms on metric $2$-step nilpotent Lie algebras are closely related to bi-invariant orthogonal complex structures. We prove that irreducible $2$-step nilpotent Lie algebras admit at most one bi-invariant complex structure, up to sign. We obtain that the space of Killing $2$-forms in such Lie algebras is at most one-dimensional. Through the de Rham decomposition we achieve the full description of the space of left-invariant Killing $2$-forms on $(N,g)$.\medskip

Let $\alpha$ be a $2$-form on a metric $2$-step nilpotent Lie algebra $(\mn,g)$ and consider its decomposition with respect to \eqref{eq:lambdad}, so that $\alpha=\alpha_2+\alpha_1+\alpha_0$. Using the metric, the $2$-form $\alpha$ is identified with a skew-symmetric endomorphism of $\mn$. Viewed as endomorphisms, $\alpha_2$ vanishes on $\mz$ and preserves $\mv$, $\alpha_1$ maps $\mv$ to $\mz$ and $\mz$ to $\mv$, and $\alpha_0$ vanishes on $\mv$ and preserves $\mz$. The next proposition is a restatement of Corollary 3.2. in \cite{BDS}.
\begin{pro} \label{pro:kill2} With the notation above, $\alpha\in \Lambda^2\mn^*$ is a Killing $2$-form if and only if $\alpha_1=0$ and the endomorphisms $\alpha_0$ and $\alpha_2$ satisfy
\begin{equation}\label{eq:kill2}
 j(\alpha_0z)=3\alpha_2j(z)=-3j(z)\alpha_2, \quad \mbox{ for all }z\in \mz.
\end{equation}
\end{pro}

\begin{proof}
Let $\{e_1, \ldots, e_n\}$, $\{z_1, \ldots,z_m\}$ be orthonormal bases of $\mv$ and $\mz$, respectively. If $\alpha$ is a Killing $2$-form, then \eqref{eq:pp1} for $l=0$ gives
$$0=\sum_{i=1}^n[x,e_i]\wedge( x\lrcorner e_i\lrcorner \alpha_2)=[\alpha_2x,x], \quad \mbox{ for all } x\in \mv.$$
Polarizing this equality, we get that for every $x,y\in \mv$, 
$$[\alpha_2x,y]=[x,\alpha_2y].$$
By \eqref{eq:jota}, this equality is equivalent to
\begin{equation}
\label{eq:comm}
j(z)\alpha_2+\alpha_2j(z)=0,\quad  \mbox{ for all }z\in \mz. 
\end{equation}
In addition, \eqref{eq:pp3} for $l=0$ gives
$$\sum_{i=1}^n[x,e_i]\lela \alpha_1z,e_i\rira=-2\alpha_1j(z)x,$$
which can be rewritten as $j(z')\alpha_1z=2j(z)\alpha_1z'$ for all $z,z'\in \mz$. Interchanging the roles of $z$ and $z'$ yields $j(z)\alpha_1(z')=0$ for all $z,z'\in \mz$ and thus $\alpha_1(\mz)$ is in the kernel of $j(z)$ for all $z\in \mz$. From Eq. \eqref{int} we obtain $\alpha_1=0$. 

Finally, taking $l=1$ in  \eqref{eq:pp3} yields
$$0=2\alpha_2j(z)x-j(z)\alpha_2x-j(\alpha_0z)x,\quad \mbox{ for all }x\in \mv$$ which together with \eqref{eq:comm} gives $3\alpha_2j(z)=j(\alpha_0z)$ for all $z\in \mz$.
 
For the converse, consider  a $2$-form $\alpha=\alpha_2+\alpha_0$ on $\mn$ satisfying \eqref{eq:kill2}. Eq. \eqref{eq:pp2} for $l=0$ holds trivially. Moreover, from the computations above it is clear that $\alpha$ verifies \eqref{eq:pp3} for $l=0,1$ and
\eqref{eq:pp1} for $l=0$. Since $\alpha_1=0$, \eqref{eq:pp1} and \eqref{eq:pp2} are trivially satisfied for $l=1$. Hence $\alpha$ is a Killing form by Corollary \ref{cor:killgen}.
\end{proof}

Consider the decomposition of $\mn$ as an orthogonal direct sum of ideals $\mn=\ma\oplus\mn_0$, as in Remark \ref{an}, where $\ma$ is the kernel of $j:\mz\lra \so(\mv)$.

\begin{pro}\label{pro:sina} Every Killing $2$-form on $\mn$ is the sum of a Killing $2$-form on $\mn_0$ and a $2$-form on $\ma$.
\end{pro}

\begin{proof}
Let $\alpha=\alpha_2+\alpha_0$ be a Killing $2$-form on $\mn$. By \eqref{eq:kill2}, $\alpha_0$ preserves $\ma=\ker j$ and, being skew-symmetric, it also preserves its orthogonal $\mn_0$ in $\mn$. Hence $\alpha$ is the sum of a 2-form in $\Lambda^2\ma^*$ (which is automatically parallel)
and a $2$-form in $\Lambda^2\mn_0^*$ which is Killing on $\mn_0$.
\end{proof}
As a consequence of this proposition we can focus the study of Killing $2$-forms on $2$-step nilpotent Lie algebras for which $j:\mz\lra \so(\mv)$ is  injective. We assume this condition for the rest of the section.\medskip

\begin{defi}\label{irr}
A metric nilpotent Lie algebra $(\mn,\bil)$ is called {\em reducible} if it can be written as an orthogonal direct sum of ideals $\mn=\mn'\oplus\mn''$. In this case we consider $(\mn',\bil')$ and $(\mn'',g'')$ as metric Lie algebras, where $\bil', \bil''$ are the restriction of $\bil$ to $\mn',\mn''$, respectively. Otherwise, $(\mn,g)$ is called {\em irreducible}. 
\end{defi}

\begin{pro} \label{pro:dec2} Suppose $(\mn,g)$ is reducible and write $\mn=\mn'\oplus\mn''$ as an orthogonal direct sum of ideals. Then any Killing $2$-form $\alpha$ on $\mn$ can be written
as $\alpha=\alpha'+\alpha''$ where $\alpha'$ (resp. $\alpha''$) is a Killing $2$-form on $\mn'$ (resp. $\mn''$).
\end{pro}

\begin{proof} Let $\alpha$ be a Killing $2$-form on $\mn$. By Proposition \ref{pro:kill2}, $\alpha=\alpha_2+\alpha_0$ where $\alpha_2\in \so(\mv)$ and $\alpha_0\in \so(\mz)$ satisfy \eqref{eq:kill2}, which is equivalent to 
\begin{equation}\label{eq:alphaxy}
\alpha_0[x,y]=3[\alpha_2x,y]=3[x,\alpha_2y], \quad \mbox{for all }x,y\in \mv.
\end{equation}

The hypothesis implies that $\mv=\mv'\oplus\mv''$ and $\mz=\mz'\oplus\mz''$ where $\mn'=\mv'\oplus\mz'$ and $\mn''=\mv''\oplus\mz''$. Let $x\in \mv'$, then for any $y\in \mv''$ we have
$3[\alpha_2x,y]=\alpha_0[x,y]=0$. Therefore $\alpha_2x\in \mv'$ so $\alpha_2$ preserves $\mv'$ and hence its orthogonal in $\mv$, namely $\mv''$. 

Since $j$ in injective, we have $\mz=[\mv,\mv]=[\mv',\mv']+[\mv'',\mv'']\subset \mz'+\mz''$, whence $[\mv',\mv']=\mz'$ and $[\mv'',\mv'']=\mz''$. By \eqref{eq:alphaxy} we immediately obtain that $\alpha_0$ preserves $\mz'$ and $\mz''$.

Thus $\alpha=\alpha'+\alpha''$ where $\alpha'$ and $\alpha''$ are Killing $2$-forms on $\mn'$ and $\mn''$, respectively, by Proposition \ref{pro:kill2}.
\end{proof}

Along the rest of the section, we focus on irreducible $2$-step nilpotent Lie algebras and show that non-zero Killing $2$-forms appear only when the Lie algebra is complex.

\begin{defi} \label{def:bii} A bi-invariant complex structure on a Lie algebra $\mn$ is an endomorphism $J$ of $\mn$ which verifies $J^2=-1$ and 
\begin{equation}\label{j}
 J[x,y]=[Jx,y] \qquad \mbox{ for all } x,y \in \mn. 
\end{equation}
\end{defi}
Notice that $\mn$ admits a bi-invariant complex structure if and only if $(\mn,J)$ is a complex Lie algebra.

\begin{lm} \label{mm} Let $S$ be a symmetric endomorphism of $\mn$ satisfying 
\begin{equation}\label{m}S[x,y]=[Sx,y]\qquad \mbox{ for all } x,y \in \mn. 
\end{equation}
If $\mn$ is irreducible, then $S$ is a multiple of the identity. 
\end{lm}
\begin{proof}
Eq. \eqref{m} shows that $Sx\in\mz$ for every $x\in\mz$, so $S$ preserves $\mz$. Since $S$ is symmetric, it also preserves its orthogonal $\mv$. 

Let $\mv=\mv_1\oplus \ldots\oplus  \mv_s$ be the orthogonal decomposition of $\mv$ in eigenspaces of $S|_\mv$, so that $S|_{\mv_i}=\lambda_i {\rm Id}_{\mv_i}$ where $\lambda_1, \ldots, \lambda_s$ are distinct real numbers. Notice that $[\mv_i,\mv_j]=0$ if $i\neq j$, indeed for $x\in \mv_i$, $y\in \mv_j$, \eqref{m} implies
$$\lambda_i[x,y]=[Sx,y]=[x,Sy]=\lambda_j[x,y].$$

For each $i=1, \ldots, s$, define $\mz_i:=[\mv_i,\mv_i]$. From the above and the fact that $j$ is injective, we obtain
\begin{equation}\label{eq:center1}
\mz=[\mv,\mv]=\sum_{i=1}^s[\mv_i,\mv_i]= \sum_{i=1}^s\mz_i.
\end{equation}
Eq. \eqref{m} implies that $S|_{\mz_i}=\lambda_i{\rm Id}_{\mz_i}.$
Since the eigenspaces of $S$ are mutually orthogonal, the sum in \eqref{eq:center1} is an orthogonal direct sum. 

It is clear that $\mn_i:=\mv_i\oplus\mz_i$ is an ideal of $\mn$ for each $i=1, \ldots, s$ and also $\mn_i\bot \mn_j$ when $i\neq j$. Since $\mn$ is irreducible, we obtain that $s=1$, and thus $S=\lambda_1{\rm Id}_\mn$.
\end{proof}

\begin{pro} \label{pro:kb} On an irreducible $2$-step nilpotent metric Lie algebra, there is a one-to-one correspondence between non-zero Killing $2$-forms, up to multiplication with a positive real number, and orthogonal bi-invariant complex structures. \end{pro}

\begin{proof}
Let $\alpha=\alpha_2+\alpha_0$ be a non-zero Killing $2$-form on $\mn$ as in Proposition \ref{pro:kill2}, where $\alpha_2\in \so(\mv)$ and $\alpha_0\in \so(\mz)$. Let $S$ be the endomorphism of $\mn$ which preserves $\mz$ and $\mv$ and verifies $S|_{\mv}=\alpha_2^2$, $S|_{\mz}=\frac19\alpha_0^2$. Since $\alpha_0$ and $\alpha_2$ verify \eqref{eq:alphaxy}, we get that $S$ is a symmetric endomorphism of $\mn$ which satisfies \eqref{m}. Lemma \ref{mm} implies that $S$ is a multiple of the identity $S=\lambda {\rm Id}_\mn$. If $\alpha$ is non-zero, $S$ is also non-zero, so necessarily $\lambda<0$. 

It is easy to check that the endomorphism $J\in \so(\mn)$ defined as
$$J|_{\mv}=\frac1{\sqrt{-\lambda}}\alpha_2, \quad \mbox{ and }\quad J|_{\mz}=\frac1{3\sqrt{-\lambda}}\alpha_0$$ is a bi-invariant complex structure on $\mn$.

For the converse, let $J$ be an orthogonal bi-invariant complex structure on $\mn$. The bi-invariance of $J$ implies that it preserves $\mv$ and $\mz$. Then the $2$-form $\alpha=\alpha_2+\alpha_0\in\mathfrak{so}(\mv)\oplus\mathfrak{so}(\mz)\subset \mathfrak{so}(\mn)$ defined by $\alpha_2:=J|_\mv$ and $\alpha_0:=3J|_\mz$ satisfies \eqref{eq:alphaxy} because of \eqref{j}, so $\alpha$ is a Killing $2$-form in view of Proposition \ref{pro:kill2}. 

Finally, it is clear that up to multiplication of $\alpha$ with a positive real number, the above constructions are inverse to each other.
\end{proof}


\begin{remark} \label{p1} Notice that the Killing 2-forms constructed above starting from bi-invariant orthogonal complex structures are never parallel. Indeed, given a bi-invariant orthogonal complex structure $J$ on $\mn$, and taking the scalar product in \eqref{j} with any $z\in\mz$, we get $j(z) J=-j(Jz)$ for every $z\in\mz$. Transposing this equality and using the skew-symmetry of $j(z)$, $j(Jz)$ and $J$, we obtain that $J$ anti-commutes with $j(z)$ for every $z\in\mz$. Therefore $\nabla_z\alpha=-\frac12[j(z),\alpha]=-\frac12[j(z),\alpha_2]=-\frac12[j(z),J|_{\mv}]=-j(z)J|_\mv$ cannot vanish for all $z\in\mz$.
\end{remark}

We are thus led to study bi-invariant orthogonal complex structures on irreducible $2$-step nilpotent metric Lie algebras. The next result shows that the situation is quite simple:

\begin{pro} \label{pro:un} An irreducible $2$-step nilpotent metric Lie algebra admits (up to sign) at most one bi-invariant orthogonal complex structure.
\end{pro}

\begin{proof}
Assume that $J$ is a bi-invariant orthogonal complex structure on $\mn$. The argument in the previous remark shows that $J$ anti-commutes with $j(z)$ for every $z\in\mz$. If $I$ is another bi-invariant orthogonal complex structure on $\mn$, it also has to anti-commute with $j(z)$ for every $z\in\mz$, whence the skew-symmetric endomorphism $IJ-JI$ commutes with $j(z)$ for every $z\in\mz$.

On the other hand, for every $x,y\in \mn$ we have  by \eqref{def:bii}:
$$[IJx,y]=I[Jx,y]=IJ[x,y],$$
which after taking the scalar product with any $z\in\mz$ reads $j(z)IJ=j(JIz)$, and similarly $j(z)JI=j(IJz)$. Subtracting these two equalities yields 
$$j(z)(IJ-JI)=-j((IJ-JI)z)\qquad \mbox{ for all } z \in \mz. $$
However, since $IJ-JI$ commutes with $j(z)$, the left hand term is symmetric, whereas the right hand term is skew-symmetric. Thus they both vanish for all $z\in\mz$. As $j$ is injective, this shows that $IJ-JI$ vanishes on $\mz$, and since the endomorphisms $j(z)$ have no common kernel in $\mv$, $IJ-JI$ also vanishes on $\mv$. Thus $I$ and $J$ commute, so $S:=IJ$ is symmetric. Moreover $S$ also satisfies \eqref{m}, so by Lemma \ref{mm}, $IJ$ is a multiple of the identity. As $(IJ)^2=\rm Id_\mn$, we necessarily have $I=\pm J$.
\end{proof}

In \cite{dBM} it is showed that each de Rham factor of $(N,g)$ is again a $2$-step nilpotent Lie group endowed with a left-invariant metric. Namely, the de Rham decomposition of $(N,g)$ corresponds to the decomposition of $(\mn,g)$ into irreducible orthogonal ideals:
\begin{pro}\cite[Corollary A.4]{dBM}\label{dR} Let $(\mn,g)$ be a $2$-step nilpotent metric Lie algebra. Then there exist {\em irreducible} $2$-step nilpotent metric Lie algebras $(\mn_i,g_i)$ $i\in\{1,\ldots,q\}$ (unique up to reordering) such that 
$$(\mn,g)=(\ma,g_0)\oplus \bigoplus_{i=1}^q(\mn_i,g_i),$$
for some abelian metric Lie algebra $(\ma,g_0)$.
\end{pro} 
We are now ready to state the main result of this section.

\begin{teo}\label{t2}
Let $(N,g)$ be a simply connected $2$-step nilpotent Lie group endowed with a left-invariant Riemannian metric. Then any invariant Killing $2$-form is the sum of left-invariant Killing $2$-forms on its de Rham factors. Moreover, the dimension of $\mathcal K^2(N,g)$, the space of left-invariant Killing $2$-forms on $(N,g)$, is
\begin{equation}
\dim \mathcal K^2(N,g)=\frac{d(d-1)}2+r, \label{teo2}
\end{equation}
where $d$ is the dimension of the Euclidean factor in the de Rham decomposition of $(N,g)$, and $r$ is the number of irreducible de Rham factors admitting bi-invariant orthogonal complex structures.
\end{teo}

\begin{proof}
Consider the de Rham decomposition of $(N,g)$ 
$$ (\R^d,g_0)\times (N_1,g_1)\times \ldots \times (N_q,g_q)$$
where $\R^d$ is the Euclidean factor, and let $(\mn,g)=(\ma,g_0)\oplus\bigoplus_{i=1}^q (\mn_i,g_i)$ be the corresponding decomposition of its Lie algebra  as in Proposition \ref{dR}.

Propositions \ref{pro:sina} and \ref{pro:dec2} imply that any Killing $2$-form on $\mn$ is the sum of Killing $2$-forms on each of the factors $\ma$ and $\mn_i$, $i=1, \ldots,q$, and conversely. Thus the space of invariant Killing $2$-forms on $(N,g)$ is the direct sum of the spaces of invariant Killing $2$-forms on $(\R^d,g_0)$ and $(N_i,g_i)$. 

Any left-invariant differential $2$-form on $\R^d$ is parallel, thus Killing, and according to Propositions \ref{pro:kb} and \ref{pro:un}, the space of left-invariant Killing $2$-forms on $(N_i,g_i)$ is one-dimensional if $(\mn_i,g_i)$ admits a bi-invariant orthogonal complex structure and zero otherwise.
\end{proof}

Note that by slightly changing the point of view, it might be interesting to ask the following question: {\em Which simply connected 2-step nilpotent Lie groups carry a left-invariant Riemannian metric admitting non-zero Killing $2$-forms?} Using Theorem \ref{t2} we can give a complete answer in low dimensions. To address this question, let us first introduce some terminology.

\begin{defi}\label{def:factor}
An ideal $\mh$ of a nilpotent Lie algebra $\mn$ is a factor of $\mn$ if there exist another ideal $\tilde \mh$ of $\mn$ such that $\mn=\mh\oplus \tilde \mh$.
\end{defi}

Note that in the decomposition of $\mn$ in Proposition \ref{dR}, each $\mn_i$ and $\ma$ are factors  of $\mn$.

\begin{cor}\label{fac2}
A simply connected 2-step nilpotent Lie group $N$, with corresponding Lie algebra $\mn$,  admits a left-invariant  Riemannian metric $g$ such that $(N,g)$ carries non-zero Killing $2$-forms if and only if one of the following (non-exclusive) conditions holds:
\begin{enumerate}
\item $\mn$ has an abelian $2$-dimensional factor,
\item $\mn$ has a non-abelian factor admitting a bi-invariant complex structure.
\end{enumerate}
\end{cor}
\begin{proof}
If there exists a left-invariant Riemannian metric $g$ such that $\dim \mathcal K^2(N,g)\geq 1$ then, according to \eqref{teo2} we have that either $d\geq 2$, i.e., the Euclidean factor of $(N,g)$ has dimension at least $2$ or $r\geq 1$, that is, $\mn$ has a factor $\mn_i$ which carries a bi-invariant complex structure.

Conversely, if $\mn$ has a $2$-dimensional abelian factor so that $\mn=\R^2\oplus \mh$ we define an inner product of $\mn$ adding the standard inner product of $\R^2$ and an arbitrary inner product on the ideal $\mh$. The corresponding left-invariant Riemannian metric on $N$ has an Euclidean factor of dimension at least $2$, thus $\dim \mathcal K^2(N,g)\geq 1$. Finally, suppose that $\mn$ is a direct sum of ideals $\mn=\mh\oplus\tilde\mh$ and $\mh$ admits a bi-invariant complex structure $J$. Consider any inner product $h$ on $\mh$, then $J$ is orthogonal with respect to the inner product $h(\cdot,\cdot)+h(J\cdot, J\cdot)$. Extending this to an inner product $g$ on $\mn$ by adding any inner product on $\tilde \mh$, we get that $r\geq 1$ in \eqref{teo2} for $(N,g)$.
\end{proof}

A real Lie algebra of dimension $2p$ carrying a bi-invariant complex structure can be seen as a complex Lie algebra of dimension $p$. Up to complex dimension $4$, there are only two non-abelian 2-step nilpotent complex Lie algebras: the complex Heisenberg Lie algebra $\mh_3^\C$ in dimension $3$ and $\C\oplus\mh_3^\C$ in dimension 4. From Corollary \ref{fac2} and Remark \ref{class}, we obtain the following classification result.

\begin{teo} There exist exactly $14$ isomorphism classes of (non-abelian) $2$-step nilpotent Lie algebras of dimension $p\leq 8$ admitting an inner product for which the corresponding simply connected Riemannian Lie group carries non-zero Killing $2$-forms:
\begin{itemize}
\item $p=5$: $\R^2\oplus \mh_3$;
\item $p=6$: $\R^3\oplus \mh_3$ and $\mh_3^\C$;
\item $p=7$:  $\R\oplus \mh_3^\C$ and $\R^2\oplus \mh$, where $\mh\in \mathcal N_5$;
\item $p=8$: $\R^2\oplus \mh$, where $\mh\in \mathcal N_6$.
\end{itemize}
\end{teo}
In the above list $\mh_3^\C$ represents the real Lie underlying the complex Heisenberg Lie algebra.

The next example shows that on a given nilpotent Lie algebra there might exist several non-equivalent inner products for which the corresponding Riemannian Lie group carries non-zero Killing 2-forms.

\begin{ex} \label{ex1} The real Lie algebra underlying $\mh_3^\C$ admits a basis $\{e_1,e_2,e_3,e_4,z_1,z_2\}$
satisfying the bracket relations
\begin{equation*}
[e_1,e_3]= z_1=-[e_2,e_4] \quad \mbox{ and } \quad [e_2,e_3]= z_2=[e_1,e_4].
\end{equation*}

Consider the one parameter family of metrics $g_\lambda$, with $\lambda>0$, for which the basis $\{e_1,e_2,e_3,e_4,z_1/\lambda,z_2/\lambda\}$ is orthonormal.
The complex structure defined by $Je_1=e_2$, $Je_3=e_4$ and $Jz_1=z_2$ is bi-invariant and orthonormal with respect to each $g_\lambda$. Thus, $(N,g_\lambda)$ admits non-zero Killing 2-forms for every $\lambda>0$ by Theorem \ref{t2}.

The center $\mz$  of $\mn$ is spanned by $z_1,z_2$, being its orthogonal $\mv={\rm span}\{e_1, e_2,e_3, e_4\}$. The map $j_\lambda:\mz\lra \so(\mv)$ associated to the metric $g_\lambda$ verifies ${\rm tr}\,(j_\lambda(z)^2)=-4\lambda^2\,g_\lambda( z,z)$ for every $z\in \mz$. 

Wilson in  \cite{Wi82} proved that two simply connected Riemannian Lie groups endowed with left-invariant metrics are isometric Riemannian manifolds if and only if their Lie algebras are isometrically isomorphic. Suppose that there exists an isometric  isomorphism $f:(\mn,g_\lambda)\lra (\mn,g_\mu)$. Then $f$ preserves $\mz$, because it is an isomorphism, and it also preserves $\mv$, since it is an isometry. Moreover, it is easy to check that
\begin{equation}
j_\mu(f(z))=f j_\lambda(z) f^{-1},\quad \mbox{for all }z\in \mz.
\end{equation}
This implies, for every $z\in\mz$, 
$$4\mu^2\,g_\lambda( z,z)=4\mu^2\,g_\mu( f(z),f(z)) =-{\rm tr}\,( j_\mu(f(z))^2)=-{\rm tr}\, (j_\lambda(z)^2)=4\lambda^2\,g_\lambda( z,z)$$ so $\lambda=\mu$. 
\end{ex}

\begin{remark}
The correspondence between left-invariant Killing 2-forms and bi-invari\-ant orthogonal complex structures in the irreducible case (Proposition \ref{pro:kb}) was obtained in a slightly different form by A. Andrada and I. Dotti (cf. \cite{AD19}, Remark 3.8). They also showed that the space of Killing 2-forms on a 2-step nilpotent metric Lie algebra $(\mn,g)$ is at most one-dimensional if $\mn$ is a Lie algebra associated to a graph (\cite{AD19}, Theorem 4.1).

The difference between our approach and the one in \cite{AD19} is basically due to the reduction procedure. Our strategy is to use the de Rham decomposition of the Riemannian Lie group in order to reduce the problem to the irreducible case, whereas in \cite{AD19}, the authors consider the spectral decomposition associated to the  Killing 2-form, which in general contains less information than the de Rham decomposition. This subtle but essential difference allowed us to prove Theorem \ref{t2}, generalizing the results in \cite{AD19}.



\end{remark}

\section{Invariant Killing $3$-forms on $2$-step nilpotent Lie groups}

In this section we prove that non-zero left-invariant Killing $3$-forms on irreducible $2$-step nilpotent metric Lie algebras only exist when the corresponding Lie group is naturally reductive as homogeneous space under its isometry group and, in this case, the space of such Killing forms is one-dimensional. As in the case of Killing $2$-forms, we obtain a description of the space of left-invariant Killing $3$-forms on Riemannian Lie groups using their de Rham decomposition.\medskip

We maintain the notation $(N,g)$ for a $2$-step nilpotent Lie group endowed with a left-invariant metric and $\mn$ the corresponding Lie algebra which we write as orthogonal direct sum $\mn=\mv\oplus\mz$. Consider $\alpha\in \Lambda^3\mn^*$ and decompose it like in \eqref{eq:lambdad} as
$ \alpha=\alpha_3+\alpha_2+\alpha_1+\alpha_0$, where $\alpha_l$ is the projection of $\alpha$ on $\Lambda^l\mv^*\otimes \Lambda^{3-l}\mz^*$. 

In Corollary \ref{cor:killgen} we have seen that the Killing condition for a $k$-form $\alpha$ imposes restrictions on its components $\alpha_l$, $l=0,\ldots, k$. The next two results show that when $k=3$, these restrictions have a simpler interpretation.

\begin{pro} 
If $\alpha $ is a Killing $3$-form then $\alpha_{1}=\alpha_{3}=0$.
\end{pro}
\begin{proof} Let $\{e_1, \ldots, e_n\}$ and $\{z_1, \ldots, z_m\}$ be orthonormal bases of $\mv$ and $\mz$, respectively. If $\alpha$  is a Killing form, then \eqref{eq:pp3} for $l=0$ gives:
$$ \sum_{i=1}^n [x,e_i]\wedge (z\lrcorner \ e_i\lrcorner\ \alpha_{1})=2(j(z)x)\lrcorner \alpha_{1},  \quad \mbox{ for all } x\in \mv, z\in \mz.$$
By viewing $\alpha_{1}$ as a linear map $\alpha_1:\mv\lra \so(\mz)$, this is an equality between $2$-forms on $\mz$ and the left hand side is $\sum_{i=1}^n [x,e_i]\wedge (z\lrcorner  \alpha_1(e_i))$. Evaluating both sides on $z',z''\in \mz$ we obtain
\begin{eqnarray}\label{eq:A}
 \lela \alpha_1(j(z')x)z,z''\rira- \lela\alpha_1( j(z'')x)z,z'\rira&=&2 \lela\alpha_1( j(z)x)z',z''\rira.
\end{eqnarray}
 
Let us denote by $A(z,z',z''):=\lela \alpha_1(j(z)x)z',z''\rira$. Then $A(z,z'',z')=-A(z,z',z'')$ and \eqref{eq:A} can be written as
$ A(z',z,z'')+ A(z'',z',z)=2A(z,z',z'')$. Interchanging $z$ with $z'$ we also have $ A(z,z',z'')+ A(z'',z,z')=2A(z,'z,z'')$. Adding these two equations we get $A(z,z',z'')=-A(z',z,z'')$ which together with \eqref{eq:A} implies $A=0$. Therefore $(j(z)x)\lrcorner \alpha_{1}=0$ for all $z\in\mz$ and $x\in \mv$. Using Lemma \ref{lm:imjzv}, we obtain $\alpha_{1}=0$.

To prove that $\alpha_3$ vanishes we make use of \eqref{eq:pp3} for $l=2$. This gives, for all $x\in \mv$ and $z\in \mz$,
\begin{eqnarray*}
2 (j(z)x)\lrcorner \alpha_{3}&=&- \sum_{i=1}^n j(z)e_i\wedge (x\lrcorner \ e_i\lrcorner\ \alpha_{3})\\
 &=& \sum_{i=1}^n x\lrcorner \ (j(z)e_i\wedge e_i\lrcorner\ \alpha_{3})- \sum_{i=1}^n \lela x,j(z)e_i\rira e_i\lrcorner \ \alpha_{3}\\
 &=& x\lrcorner \ j(z)( \alpha_{3})+(j(z)x)\lrcorner\ \alpha_{3}.
\end{eqnarray*} 

Hence $(j(z)x)\lrcorner \alpha_{3}= x\lrcorner j(z)( \alpha_{3})$ for all $x\in \mv$ and $z\in \mz$, where we denoted by $ j(z)( \alpha_{3})$ the action of the skew-symmetric endomorphism $j(z)$ of $\mv$ on $\alpha_3$, when extended as a derivation of the exterior algebra of $\mv^*$ like in \eqref{der}. By taking $x=e_i$ in this formula, making the wedge product with $e_i$ and summing over $i$, we get
$$\sum_{i=1}^n e_i\wedge (j(z)e_i)\lrcorner \alpha_{3}= \sum_{i=1}^n e_i \wedge e_i\lrcorner  j(z)( \alpha_{3}).$$ The right hand term is equal to $3 j(z)(\alpha_{3})$, whereas the left hand term is equal to $-j(z)(\alpha_{3})$. Therefore we must have $(j(z)x)\lrcorner \alpha_{3}=0$ for all $x\in \mv$ and $z\in \mz$, which implies $\alpha_{3}=0$ in view of Lemma \ref{lm:imjzv}.
\end{proof}

We conclude that any Killing $3$-form $\alpha\in \Lambda^3\mn^*$ has only two non-zero components, namely, $\alpha_2$ and $\alpha_0$. To simplify the notation, we denote from now on these components as $\beta$ and $\gamma$, respectively. Thus any Killing $3$-form on $\mn$ is of the form
\begin{equation}
\alpha=\beta  +\gamma, \quad \mbox{ where }\beta\in \Lambda^2\mv^*\otimes \mz^* \mbox{ and } \gamma\in \Lambda^3\mz^*.\label{eq:f6}
\end{equation} 

By using the metric, we see $\beta$ as a linear map $\beta:\mz\lra \so(\mv)$ and $\gamma$ as a skew-symmetric bilinear map $\gamma:\mz\times \mz\lra \mz$, such that for each $z\in \mz$, $\gamma(z,\cdot)\in \so(\mv)$.

\begin{pro}\label{teo:kill} Let  $\alpha=\beta +\gamma$ be a $3$-form  as in \eqref{eq:f6} and let $\{z_1 ,\ldots, z_m\}$ be an orthonormal basis of $\mz$. Then $\alpha$ is a Killing $3$-form if and only if for all $z,z'\in \mz$ and $x,y\in \mv$ the following conditions hold:
\begin{equation}\label{gamma1} 
\beta(z)j(z')- \beta(z')j(z)+[j(z),\beta(z')]= j(\gamma(z,z'))
\end{equation}
\begin{equation}\label{gamma2} 
\sum_{t=1}^m j(z_t)x\wedge \beta(z_t) x=0.
\end{equation}
\end{pro}

\begin{proof} Assume that $\alpha=\beta+\gamma$ as in \eqref{eq:f6} is a Killing $3$-form.
Let $\{e_1,\ldots, e_n\}$ and  $\{z_1, \ldots, z_m\}$ be orthonormal bases of $\mv$ and $\mz$, respectively. Then  \eqref{eq:pp3} for $l=1$ implies that for any $x\in \mv$ and $z\in \mz$ one has
$$
\sum_{i=1}^n [x,e_i]\wedge (z\lrcorner \ e_i\lrcorner\ \beta)-2(j(z)x)\lrcorner \beta - \sum_{i=1}^n j(z)e_i\wedge (x\lrcorner \ e_i\lrcorner\ \beta)=
\sum_{t=1}^mj(z_t)x\wedge (z\lrcorner \ z_t\lrcorner\ \gamma).
$$

We evaluate each term of this equality on  $x'\in \mv$ and $z'\in \mz$. Straightforward computations give
\begin{eqnarray*}
\left\{(j(z)x)\lrcorner \beta\right\}(x',z')  &=& \lela \beta(z')j(z)x,x'\rira\\
 \sum_{i=1}^n \left\{[x,e_i]\wedge (z\lrcorner \ e_i\lrcorner\ \beta)\right\}(x',z') &=&
\sum_{i=1}^n\lela j(z')x,e_i\rira\lela \beta(z) e_i,x'\rira =\lela \beta(z)j(z')x,x'\rira\\
 \sum_{i=1}^n \left\{j(z)e_i\wedge (x\lrcorner \ e_i\lrcorner\ \beta)\right\}(x',z')&=& -\sum_{t=1}^m \lela  j(z)e_i,x'\rira\lela\beta(z')x,e_i\rira=- \lela  j(z)\beta(z')x,x'\rira\\
 \sum_{t=1}^m\left\{j(z_t)x\wedge (z\lrcorner \ z_t\lrcorner\ \gamma)\right\}(x',z')&=&   \sum_{t=1}^m\lela [x,x'],z_t\rira  \  \lela\gamma(z,z'),z_t\rira= \lela\gamma(z,z'),[x,x']\rira .
\end{eqnarray*}

Therefore, \eqref{eq:pp3} for $l=1$ can be written in a simpler form as
$$\lela  (\beta(z)j(z')-2 \beta(z')j(z)+j(z)\beta(z'))x,x'\rira=\lela j(\gamma(z,z'))x,x'\rira, $$
which is equivalent to
$$\beta(z)j(z')- \beta(z')j(z)+[j(z),\beta(z')]= j(\gamma(z,z')),\qquad \mbox{ for all }z,z'\in \mz,$$ and hence we get \eqref{gamma1}. Finally, \eqref{eq:pp1} for $l=2$ gives for every $x\in \mv$
$$0 = \sum_{t=1}^mj(z_t)x\wedge (x\lrcorner \ z_t\lrcorner\ \beta)=\sum_{t=1}^m j(z_t)x\wedge \beta(z_t) x,$$ which is \eqref{gamma2}.\smallskip

Conversely, let $\alpha=\beta+\gamma$ be a $3$-form as in \eqref{eq:f6} and suppose it satisfies \eqref{gamma1} and \eqref{gamma2}. We shall prove that \eqref{eq:pp1}--\eqref{eq:pp3} hold for every $l\in\{0,1,2\}$. From the reasoning above it is easy to check that \eqref{gamma1} implies that \eqref{eq:pp3} holds for $l=1$, and \eqref{gamma2} gives \eqref{eq:pp1} for $l=2$.
Also, note that \eqref{eq:pp1} for $l=1$, \eqref{eq:pp2} for $l=0,1$ and \eqref{eq:pp3} for $l=0,2$ are trivially satisfied because $\alpha_1= \alpha_3=0$. 

Eq. \eqref{eq:pp1} for $l=0$ is equivalent to $\sum_{i=1}^n  [x,e_i]\wedge (x\lrcorner\ e_i\lrcorner\  \beta)=0$, where the left hand side is a 2-form on $\mz$. This form evaluated on $z,z'\in \mz$ reads
\[\sum_{i=1}^n \{ [x,e_i]\wedge (x\lrcorner\ e_i\lrcorner\  \beta)\}(z,z')=\lela (j(z)\beta(z')-j(z')\beta(z))x,x\rira
.
\]

On the other hand, \eqref{gamma1} implies that $\beta(z)j(z')- \beta(z')j(z)$ is a skew-symmetric endomorphism of $\mv$, therefore the equation above vanishes and \eqref{eq:pp1} for $l=0$ is verified. 

Finally, notice that the left hand side of \eqref{eq:pp2} for $l=2$ is
$$\sum_{i=1}^nj(z)e_i\wedge (z\lrcorner\ e_i\lrcorner\ \beta)=-\sum_{i=1}^nj(z)e_i \wedge \beta(z)e_i.$$

This $2$-form, viewed as skew-symmetric endomorphism, is actually the commutator $[\beta(z),j(z)]$ in $\mathfrak{so}(\mv)$, which is zero for all $z\in \mz$, in view of \eqref{gamma1} applied to $z=z'$. Thus \eqref{eq:pp2} also holds for $l=2$.
\end{proof}

Consider again the decomposition of $\mn$ as an orthogonal direct sum of ideals $\mn=\ma\oplus\mn_0$, as in Remark \ref{an}, where $\ma$ is the kernel of $j:\mz\lra \so(\mv)$.

\begin{pro} \label{pro:abf} Every Killing $3$-form on $\mn$ is a sum of a Killing $3$-form on $\mn_0$ and a $3$-form in $\Lambda^3\ma^*$ (which is automatically parallel).
\end{pro}

\begin{proof} Let $\alpha=\beta+\gamma$ be a Killing $3$-form on $\mn$ as in \eqref{eq:f6}. For every $z,z'\in\ma$, \eqref{gamma1} shows that $\gamma(z,z')\in\ma$. Note that by interchanging the roles of $z,z'$ in \eqref{gamma1} we readily obtain
\begin{equation}\label{jp}[j(z),\beta(z')]+[j(z'),\beta(z)]=0, \qquad \mbox{ for all }z,z'\in \mz.\end{equation}

Consider now some arbitrary elements $z\in\ma$ and $z'\in\ma^\perp$,  where $\ma^\bot$ is the orthogonal of $\ma$ in $\mz$. Eq. \eqref{gamma1} shows that $\beta(z)j(z')=j(\gamma(z,z'))$ is skew-symmetric. On the other hand, \eqref{jp} yields $[j(z'),\beta(z)]=0$, so $j(z')$ commutes with $\beta(z)$, and therefore their composition is symmetric. Thus $j(\gamma(z,z'))=0$ for all $z\in\ma$ and $z'\in\ma^\perp$, showing that $\gamma=\gamma_0+\gamma_1$, with $\gamma_0\in\Lambda^3(\ma^\perp)^*$ and $\gamma_1\in\Lambda^3\ma^*$. In particular $\gamma_1$ is parallel, and $\alpha-\gamma_1$ is a Killing $3$-form on $\mn_0$.
\end{proof}

The above result shows that one can reduce the study of Killing $3$-forms on $2$-step nilpotent Lie algebras to the case where $j$ is injective, which we will implicitly assume in the sequel. This ensures that $\mn$ has no abelian factor.
\medskip

In order to interpret the relation \eqref{gamma2}, we need the following general result:

\begin{pro}\label{ab} Let $\mv$ be an Euclidean space and let $A_1,\ldots,A_m,D_1,\ldots,D_m\in \so(\mv)$ be skew-symmetric endomorphisms such that $A_1,\ldots,A_m$ are linearly independent and 
\begin{equation}\label{wedge}\sum_{t=1}^mA_tx\wedge D_tx=0,\qquad\forall x\in\mv.
\end{equation}
Then each $D_t$ is a linear combination of $A_1,\ldots,A_m$.
\end{pro}
\begin{proof} By polarization we get 
\begin{equation}\label{wedgep}\sum_{t=1}^m(A_tx\wedge D_ty+A_ty\wedge D_tx)=0,\qquad\forall x,y\in\mv.
\end{equation}
Taking $s\in \{1, \ldots, m\}$ and contracting this equation with $A_sx$, yields for every $ x,y\in\mv$:
$$\sum_{t=1}^m(g(A_sx,A_tx)D_ty-g(A_sx,D_ty)A_tx+g(A_sx,A_ty)D_tx-g(A_sx,D_tx)A_ty)=0.$$
Let $\{e_1,\ldots,e_n\}$ be a $g$-orthonormal basis of $\mv$. Taking $x=e_i$ in the previous equation and summing over $i$, we obtain for every $y\in\mv$:
$$\sum_{t=1}^m(\langle A_s,A_t\rangle D_ty+A_tA_sD_ty-D_tA_sA_ty-\langle A_s,D_t\rangle A_ty)=0,$$
where $\langle\cdot,\cdot\rangle$ denotes the standard scalar product in $\so(\mv)$.
The previous relation reads
$$\sum_{t=1}^m(\langle A_s,A_t\rangle D_t-\langle A_s,D_t\rangle A_t)=\sum_{t=1}^m(-A_tA_sD_t+D_tA_sA_t).$$
As the right hand side in this formula is symmetric and the left hand side is skew-symmetric, they both vanish. We thus get
\begin{equation}\label{ba}\sum_{t=1}^m\langle A_s,A_t\rangle D_t=\sum_{t=1}^m\langle A_s,D_t\rangle A_t.
\end{equation}
This finishes the proof, since the Gram matrix whose coefficients are $G_{ts}:=\langle A_s,A_t\rangle$ is invertible by the assumption that $A_1,\ldots,A_m$ are linearly independent.
\end{proof}

We apply this proposition to the matrices $A_t:=j(z_t)$ and $D_t:=\beta(z_t)$ (recall that $j$ is injective so the matrices $A_t$ are linearly independent). Consider the $m\times m$ matrix $C$ with coefficients $C_{ts}:= \left\langle A_s,D_t\right\rangle$. Then \eqref{ba} reads 
\begin{equation}\label{GC}
\sum_{t=1}^m G_{ts} D_t=\sum_{t=1}^m C_{ts} A_t.
\end{equation}
As $G$ is symmetric, taking a scalar product with $A_{t'}$ for some $t'\in \{1,\ldots,m\}$ in \eqref{GC} yields 
$GC=\,^tCG$, which is equivalent to the fact that $B:=\, CG^{-1}$ is symmetric. By a slight abuse of language, we denote by the same letters the endomorphisms associated to the matrices $C$, $G$ and $B$ in the basis $\{z_1, \ldots, z_m\}$, that is,
$$ G(z_s):=\sum_{t=1}^m G_{ts}z_t,\quad  C(z_s):=\sum_{t=1}^mC_{ts}z_t,\quad  \text{and}\quad B(z_s):=\sum_{t=1}^m B_{ts}z_t.$$  Using again \eqref{GC} we obtain 
$$\beta(G(z_s))=\beta(\sum_{t=1}^m G_{ts}z_t)=\sum_{t=1}^m G_{ts}D_t=\sum_{t=1}^m {C}_{ts} A_t=\sum_{t=1}^m C_{ts} j(z_t)=j(C(z_s)),$$
so $\beta(G(z_s))=j( C(z_s))$ for every $s$. This is equivalent to $\beta=j\circ B$, for the above defined symmetric endomorphism $B$ of $\mz$. This interpretation of \eqref{gamma2} gives the following consequence of Proposition \ref{teo:kill}.

\begin{cor}\label{cor:jM} Let $\mn$ be a $2$-step nilpotent Lie algebra without abelian factors (i.e. such that $j:\mz\lra \so(\mv)$ is injective). A $3$-form $\alpha$ on $\mn$ is a Killing form if and only if 
\begin{equation}\label{eq:alMt}
\alpha=j\circ B +\gamma
\end{equation} 
where $\gamma\in \Lambda^3\mz^*$ and $B\in {\rm Sym}^2\mz$ satisfy
\begin{equation}\label{jM} j(Bz)j(z')- j(Bz')j(z)+[j(z),j(Bz')]= j(\gamma(z,z')) \mbox{ for all }z,z'\in \mz.
\end{equation}
\end{cor}
In the above formula \eqref{eq:alMt} and in the sequel, the linear map $j\circ B:\mz\to\so(\mv)$ is viewed as a $3$-form on $\mn$ via the standard embedding $\mz\otimes\so(\mv)\subset \Lambda^3\mn^*$ given by the metric and skew-symmetrization.  Note that by \eqref{eq:alMt}--\eqref{jM}, the symmetric endomorphism $B$ is non-zero whenever the Killing form $\alpha$ is non-zero.
\smallskip

Recall that a metric Lie algebra is called reducible if it can be decomposed as a direct sum of orthogonal ideals.

\begin{pro} \label{pro:alsum} Suppose $(\mn,g)$  is reducible and $\mn=\mn'\oplus \mn''$ as orthogonal direct sum of ideals. Then
\begin{enumerate}
\item  If $\alpha'$ and $\alpha''$ are Killing $3$-forms on $\mn'$ and $\mn''$, respectively, then $\alpha'+\alpha''$ is  a Killing $3$-form on $\mn$.
\item If $\alpha $ is a Killing $3$-form on $\mn$, then $\alpha=\alpha'+\alpha''$ where $\alpha'$ and $\alpha''$ are Killing $3$-forms on $\mn'$ and $\mn''$, respectively.
\end{enumerate}
\end{pro}
\begin{proof}
Since $\mn=\mn'\oplus\mn''$, we have $\mv=\mv'\oplus\mv''$ and $\mz=\mz'\oplus\mz''$ as orthogonal direct sums; also $j=j'+j''$ where $j':\mz'\lra \so(\mv')$ and $j'':\mz''\lra \so(\mv'')$. In particular $[j(z'),j(z'')]=0$ if $z'\in \mz'$ and $z''\in \mz''$.\smallskip

 (1) Let $\alpha'\in \Lambda^3(\mn')^*$, $\alpha''\in \Lambda^3(\mn'')^*$ be Killing $3$-forms, and let $B'\in {\rm Sym}^2\mz'$, $B''\in {\rm Sym}^2\mz''$, $\gamma'\in \Lambda^3(\mz')^*$, $\gamma''\in \Lambda^3(\mz'')^*$ be such that $\alpha'=j'\circ B'+\gamma'$ and $\alpha''=j''\circ B''+\gamma''$ as in Corollary \ref{cor:jM}. Define $B\in {\rm Sym}^2\mz$ by $B|_{\mz'}:=B'$ and $B|_{\mz''}:=B''$; also set $\gamma=\gamma'+\gamma''\in \Lambda^3\mz^*$. The commutation of $j(z')$ and $j(z'')$ for $z'\in \mz',z''\in \mz''$ implies that $B$ and $\gamma$ satisfy \eqref{jM} for every element in $\mz$. Therefore $\alpha:=j\circ B+\gamma$ is a Killing $3$-form on $\mn$ and clearly $\alpha=\alpha'+\alpha''$.\smallskip

 (2) Let $\alpha=j\circ B+\gamma$ be a Killing $3$-form on $\mn$; we shall prove that $B$ preserves $\mz'$ and $\mz''$. Let $z'\in \mz'$ and $z''\in \mz''$, and denote $Bz'=u'+u''$, $Bz''=v'+v''$ where $u',v'\in \mz'$ and $u'',v''\in \mz''$. Then \eqref{jM} implies
\[
0=[j(z'),j(Bz'')]+[j(z''),j(Bz')]
=[j'(z'),j'(v')]+[j''(z''),j''(u'')].
\]
The brackets in the right side of the above equation belong to $\so(\mv')$ and $\so(\mv'')$ respectively, so they both vanish: $[j'(z'),j'(v')]=0$ and $[j''(z''),j''(u'')]=0$. Using this in \eqref{jM} we get
\begin{equation}
\label{j'}
j''(u'')j''(z'')- j'(v')j'(z')= j(\gamma(z',z'')).
\end{equation}
The left hand side of this equality is symmetric while the right hand side is skew-symmetric, so they both vanish. Moreover, $j''(u'')j''(z'')$ is an endomorphism of $\mv''$ and $ j'(v')j'(z')$ is an endomorphism of $\mv'$, so again both terms are zero.

For fixed $z''\in \mz''$, we have $j'(v')j'(z')=0$ for all $z'\in\mz'$ and, in particular, $(j'(v'))^2=0$ which implies $v'=0$ (as $j'$ is injective). Similarly, we obtain $u''=0$ for $z'\in \mz'$ fixed, so $B$ preserves $\mz'$ and $\mz''$. As a consequence, and using \eqref{j'}, we see that $\gamma(z',z'')=0$ and thus $\gamma=\gamma'+\gamma''$ where $\gamma'\in \Lambda^3(\mz')^*$ and $\gamma''\in \Lambda^3(\mz'')^*$. 

Let $B'$, $B''$ denote the restrictions of $B$ to $\mz'$ and $\mz''$, respectively. Then $B'$ and $B''$ verify \eqref{jM} on their corresponding Lie algebras, $\mn'$ and $\mn''$. Moreover $\alpha':=j'\circ B'+\gamma'$ and $\alpha'':=j''\circ B''+\gamma''$ are Killing $3$-forms on $\mn'$ and $\mn''$, respectively, and $\alpha'+\alpha''=\alpha$ as claimed.
\end{proof}

Because of the previous result, we can restrict our study of Killing $3$-forms to the case where the $2$-step nilpotent Lie algebra is irreducible. To address this case we will make use of the following lemma.

\begin{lm} \label{reducn} Let $\mn=\mv\oplus\mz$ be a 2-step nilpotent Lie algebra. Let $B$ be a symmetric endomorphism of $\mz$ such that  for every $z,z'\in \mz$ it satisfies
\begin{enumerate}
\item \label{m1} $[j(z),j(Bz)]=0$ ,
\item \label{m2} $j(Bz)j(z')- j(Bz')j(z)\in \so(\mv)$. 
\end{enumerate}
If $\mn$ is reducible, then $B$ is a multiple of the identity.
\end{lm}
\begin{proof}
Let $\lambda_1, \ldots, \lambda_s$ denote the mutually distinct eigenvalues of $B$ on $\mz$, and let $\mz_i$ be the corresponding eigenspaces. For each $i=1, \ldots, s$, set
\begin{equation}
\mv_i:={\rm span}\{j(z)x:z\in\mz_i, x\in\mv\}.
\end{equation}
 Lemma \ref{lm:imjzv} implies that $\mv=\sum_{i=1}^s \mv_i$; we shall prove that this is a direct sum. We claim that for every $i,j\in \{1, \ldots,s\}$ with $i\neq j$, and $z\in \mz_i$, $z'\in \mz_j$, we have $j(z)j(z')=0$. Indeed,  polarizing \eqref{m1} and evaluating on $z,z'$ we obtain
\begin{equation}
(\lambda_j-\lambda_i) [j(z),j(z')]=0,
\end{equation} so $j(z)$ and $j(z')$ commute and thus $j(z)j(z')$ is a symmetric  endomorphism of $\mv$. On the other hand, \eqref{m2} gives
\begin{equation}\label{eq:iktau}
(\lambda_i-\lambda_j)j(z)j(z')\in \so(\mv),
\end{equation} 
so $j(z)j(z')$ is a skew-symmetric map which is also symmetric, whence $j(z)j(z')=0$ as claimed. 

Therefore  $\lela j(z)x,j(z')x'\rira=0$ for every $x,x'\in \mv$, and thus $\mv_i\bot\mv_j$, showing that $\mv=\bigoplus_{i=1}^s\mv_i$ is an orthogonal direct sum. 

Fix $i\in \{1, \ldots, s\}$ and let $z\in \mz_i$. For any $z'\in \mv_j$ with $i\neq j$ and $x\in \mv$, we have $j(z)j(z')x=0$ so $j(z)|_{\mv_j}=0$. Therefore, if $x\in \mv_i$ and $x'\in \mv_j$,  $\lela z, [x,x']\rira=0$ for all $z\in \mz$ hence $[\mv_i,\mv_j]=0$. Similarly, we obtain that $[\mv_i,\mv_i]\subset \mz_i$. Hence $\mn_i:=\mv_i\oplus\mz_i$ is an ideal of $\mn$ and $\mn_i\bot\mn_j$ if $i\neq j$. Thus, if $\mn$ is irreducible then $s=1$ and $B=\lambda_1{\rm Id}_\mz$.
\end{proof}

\begin{pro}\label{pro:onedim} An irreducible $2$-step nilpotent Lie algebra $\mn$ admits a non-zero Killing $3$-form if and only if the following conditions hold:
\begin{equation}
\left\{
\begin{array}{l}
\text{\rm (1) }j(\mz) \mbox{ is a subalgebra of }\so(\mv),\\
\text{{\rm (2)} for each }z\in \mz,\mbox{ the map }z'\mapsto j^{-1}[j(z),j(z')]\mbox{ is in }\so(\mz).
\end{array}\right.\label{nr}  
\end{equation}
In this case, the space of Killing $3$-forms on $\mn$ is one-dimensional.
\end{pro}
\begin{proof}
Let $\alpha=j\circ B+\gamma$  be a non-zero Killing $3$-form on $\mn$. By \eqref{jM}, $B$ verifies the conditions in Lemma \ref{reducn} so it is a multiple of the identity: $B=\lambda {\rm Id}_\mz$ with $\lambda\neq  0$. Then, \eqref{jM} reads
\begin{equation}
\label{eq:taul}
j(\gamma(z,z'))=2\lambda[j(z),j(z')] ,\quad \mbox{for all } z,z'\in \mz.
\end{equation}

Thus $j(\mz)$ is a subalgebra of $\so(\mv)$ and for each $z\in \mz$ fixed, the endomorphism defined as $z'\mapsto j^{-1}[j(z),j(z')]$ is a skew-symmetric map of $\mz$, since $\gamma$ is a $3$-form. Notice that $\gamma$ is uniquely determined by \eqref{eq:taul} because $j$ is injective. 

Conversely, if $j(\mz)$ is a subalgebra of $\so(\mv)$, then $\gamma_0(\cdot,\cdot)=2j^{-1}[j(\cdot),j(\cdot)]$ is well defined and it is indeed a $3$-form on $\mz$ because of (2). Moreover, by Corollary \ref{cor:jM}, 
$j+\gamma_0$ is a Killing $3$-form on $\mn$.
\end{proof}

\begin{remark}\label{p2}
A case by case analysis of the Berger-Simons Holonomy Theorem easily shows that a de Rham irreducible Riemannian manifold of dimension $\geq 4$ which carries a non-zero parallel 3-form is Einstein. Since a left-invariant metric on a 2-step nilpotent Lie group is never Einstein \cite[Proposition 2.5]{EB}, the Killing 3-forms constructed above are not parallel, except in dimension 3.
\end{remark}

The two conditions in the last proposition are closely related to the property of the corresponding simply connected Riemannian Lie group to be naturally reductive, so we recall here the relevant definitions.\medskip

A homogeneous Riemannian manifold $(M,g)$ is naturally reductive if there is a transitive group $G$ inside the full isometry group ${\rm Iso}(M,g)$ and a reductive decomposition $\mgg=\mh\oplus \mm$ of the Lie algebra $\mgg$ of $G$, such that \begin{equation*}
\lela [x,y]_{\mm},z\rira+\lela y, [x,z]_{\mm}\rira=0, \qquad \mbox{ for all }x,y,z\in \mm.
\end{equation*}

Let $(N,\bil)$ be a simply connected nilpotent Lie group endowed with a left-invariant Riemannian metric and denote by $\mn$ its Lie algebra. C. Gordon proved that if $(N,g)$ is naturally reductive, then $N$ is at most $2$-step nilpotent \cite{GO2}. Moreover, $2$-step nilpotent Lie groups which are naturally reductive can be characterized as follows (see also \cite{LA}):

\begin{teo}(cf. \cite{GO2})\label{teo:grd}
Let $(N,g)$ be a simply connected $2$-step nilpotent Lie group without Euclidean factor. Let $\mn$ denote its Lie algebra and consider the orthogonal decomposition $\mn=\mv\oplus\mz$. Then $N$ is naturally reductive if and only if $\mn$ satisfies the conditions in \eqref{nr}.
\end{teo}

This result will lead to a geometrical classification of left-invariant Killing $3$-forms on simply connected $2$-step nilpotent Lie groups. 

\begin{teo}\label{th5}
Let $(N,g)$ be a simply connected $2$-step nilpotent Lie group endowed with a left-invariant Riemannian metric. Then any invariant Killing $3$-form is the sum of left-invariant Killing $3$-forms  on its de Rham factors.
Moreover, the dimension of $\mathcal K^3(N,g)$, the space of left-invariant Killing $3$-forms on $(N,g)$,  is
\begin{equation}\label{teo3}
\dim\mathcal K^3(N,g)=\frac{d(d-1)(d-2)}6+r,
\end{equation} where $d$ and $r$ denote the dimension of the Euclidean factor and the number of naturally reductive factors in the de Rham decomposition of $(N,g)$ respectively.
\end{teo}

\begin{proof}
Consider the de Rham decomposition of $(N,g)$ 
$$ (\R^d,g_0)\times (N_1,g_1)\times \ldots \times (N_q,g_q)$$
where $\R^d$ is the Euclidean factor, and the corresponding decomposition of its Lie algebra as in Proposition \ref{dR}, $(\mn,g)=(\ma,g_0)\oplus\bigoplus_{i=1}^q (\mn_i,g_i)$.

Propositions \ref{pro:abf} and \ref{pro:alsum} imply that the space of invariant Killing $3$-forms on $(N,g)$ is the direct sum of the spaces of invariant Killing $3$-forms on $(\R^d,g_0)$ and $(N_i,g_i)$. 

Any left-invariant differential $3$-form on $\R^d$ is parallel, thus Killing, and, according to Proposition \ref{pro:onedim} and Theorem \ref{teo:grd}, the space of left-invariant Killing $3$-forms on $(N_i,g_i)$ is one-dimensional if $N_i$ is naturally reductive and zero otherwise.
\end{proof}

Like in the previous section, we might ask at this point: {\em Which simply connected 2-step nilpotent Lie groups carry a left-invariant Riemannian metric admitting non-zero Killing $3$-forms?} In order to answer this question, we need again to introduce some terminology.

\begin{defi}
A 2-step nilpotent Lie algebra $\mn$ is called of naturally reductive type if it has no abelian factor (see Definition \ref{def:factor}) and admits an inner product $g$ such that $(\mn,g)$ satisfies \eqref{nr}.
\end{defi}

Notice that the condition for $\mn$ to have no abelian factor implies that for every choice of inner product on $\mn$, the map $j$ is injective; hence \eqref{nr} makes sense.

\begin{cor}\label{fac3}
A simply connected 2-step nilpotent Lie group $N$, with corresponding Lie algebra $\mn$,  admits a left-invariant  Riemannian metric $g$ such that $(N,g)$ carries non-zero Killing $3$-forms if and only if one of the following (non-exclusive) conditions holds:
\begin{enumerate}
\item $\mn$ has an abelian $3$-dimensional factor,
\item $\mn$ has a non-abelian factor of naturally reductive type.
\end{enumerate}
\end{cor}
\begin{proof}
If there exists a left-invariant Riemannian metric $g$ such that $\dim \mathcal K^3(N,g)\geq 1$ then, according to \eqref{teo3} we have that either $d\geq 3$, i.e., the Euclidean factor of $(N,g)$ has dimension at least $3$ or $r\geq 1$, that is, $\mn$ has a factor $\mn_i$ of naturally reductive type.

Conversely, if $\mn$ has a $3$-dimensional abelian factor so that $\mn=\R^3\oplus \mh$ we define an inner product of $\mn$ adding the standard inner product of $\R^3$ and an arbitrary inner product on the ideal $\mh$. The corresponding left-invariant Riemannian metric on $N$ has an Euclidean factor of dimension at least $3$, thus $\dim \mathcal K^3(N,g)\geq 1$. 

Suppose now that $\mn$ is a direct sum of ideals $\mn=\mh\oplus\tilde\mh$ and $\mh$ is of naturally reductive type. Let $h$ be an inner product on $\mh$ such that the corresponding map $j$ satisfies \eqref{nr}.
Extending $h$ to an inner product $g$ on $\mn$ by adding any inner product on $\tilde \mh$, we obtain by Theorem \ref{th5} that the corresponding simply connected Riemannian Lie group $(N,g)$ carries a non-zero Killing $3$-form.
\end{proof}

It was noticed by Lauret \cite{LA} that 2-step nilpotent Lie algebras of naturally reductive type are related to representations of compact Lie algebras. Indeed, let  $(\mn,g)$ be a 2-step nilpotent Lie algebra endowed with an inner product such that $j:\mz\lra \so(\mv)$ is injective and satisfies  \eqref{nr}. It turns out that the bilinear map $\mz\times\mz\to\mz$ given by $[[z,z']]:=j^{-1}([j(z),j(z')])$ is a well defined Lie bracket on $\mz$ with the property
$$ \lela [[z,z']],z''\rira+\lela z',[[z,z'']]\rira=0, \quad\mbox{ for all } z,z'\in \mz.$$
Hence $(\mz,[[\, , \,]])$ is a compact Lie algebra and $j:\mz\lra \so(\mv)$ is a faithful representation such that $\bigcap_{z\in \mz}\ker j(z)=0$. Conversely, given a faithful representation $\rho:\mz\lra\so(\mv,g_{\mv})$ without trivial sub-representations  of a compact Lie algebra $\mz$, and an ad-invariant inner product $g_\mz$ on $\mz$, the vector space $\mn:=\mv\oplus \mz$ together with the inner product $g=g_\mv+g_\mz$ becomes a 2-step nilpotent metric Lie algebra by defining the Lie bracket as:
$$\lela [x,y],z\rira=g_{\mv}(\rho(z)x,y), \quad \mbox{ and } \quad [x+z,z']=0, \quad \mbox{for all }x,y\in \mv, \,z,z'\in \mz.  $$
The associated map $j$ coincides with $\rho$ and satisfies \eqref{nr}, since $g_\mz$ is ad-invariant, so $\mn$ is of naturally reductive type.

Using this approach we will now give a short description 2-step nilpotent Lie algebras of naturally reductive of dimension $\leq 6$. To this purpose, we need to study compact Lie algebras having faithful orthogonal representations of dimension $m$ without trivial factors such that $m+\dim z\leq 6$. Clearly, this implies $\dim \mz\leq 3$.

If $\dim \mz=1$ every faithful orthogonal representation without trivial sub-representations is even dimensional ($n=2l$) and for every metric on $\mz$ the corresponding 2-step nilpotent Lie algebra is isomorphic to the Heisenberg Lie algebra $\mh_{2l+1}$. In dimension $\leq 6$ we thus obtain $\mh_{3}$ and $\mh_5$. 

If $\dim\mz=2$, the unique 2-dimensional compact Lie algebra is $\R^2$, which has no faithful orthogonal representations in dimension $\le 3$ and every such representation in dimension $4$ together with any metric on $\mz$ give rise to a 2-step nilpotent Lie algebra isomorphic to $\mh_3\oplus \mh_3$.

If $\dim\mz=3$, the unique representation without trivial sub-representations of a compact 3-dimensional Lie algebra in dimension $\le 3$ is the standard representation of $\so(3)$ on $\R^3$. For every ad-invariant metric on $\mz$, the corresponding 2-step nilpotent Lie algebra is isomorphic the free 2-step nilpotent Lie algebra $\mn_{3,2}$ on $3$-generators, which has a basis $\{e_1, \ldots, e_6\}$ satisfying
$$[e_1,e_2]=e_4, \;[e_1,e_3]=e_5,\; [e_2,e_3]=e_6.$$

From Corollary \ref{fac3}, we obtain the following classification result.

\begin{teo} There exist exactly $8$ isomorphism classes of (non-abelian) $2$-step nilpotent Lie algebras of dimension $p\leq 6$ admitting an inner product for which the corresponding simply connected Riemannian Lie group carries non-zero Killing $3$-forms:
\begin{itemize}
\item $p=3$: $\mh_3$;
\item $p=4$: $\R\oplus \mh_3$;
\item $p=5$: $\R^2\oplus \mh_3$ and $\mh_5$;
\item $p=6$: $\R^3\oplus \mh_3$, $\mh_3\oplus \mh_3$, $\R\oplus \mh_5$, $\mn_{3,2}$.
\end{itemize}
\end{teo}

Like in Example \ref{ex1} one can construct, by using Theorem 4.3 in \cite{LA}, families of non-isometric left-invariant Riemannian metrics with non-zero Killing 3-forms on each of the above Lie groups.

We have seen in Theorems \ref{t2} and \ref{th5} that for $k=2$ and $k=3$, any Killing $k$-form on a 2-step nilpotent Riemannian Lie group is a sum of Killing $k$-forms on the de Rham factors. 
Our next result shows that a non-flat irreducible de Rham factor cannot carry both Killing $2$-forms and Killing $3$-forms. 

\begin{pro} \label{p4} A $2$-step nilpotent Lie group $N$ endowed with a left-invariant metric $g$ which is de Rham irreducible cannot admit non-zero Killing $2$-forms and non-zero Killing $3$-forms simultaneously.
\end{pro}
\begin{proof}
If $(N,g)$ admits non-zero Killing 2-forms, then by Propositions \ref{pro:kb}, its Lie algebra $\mn$ admits an orthogonal bi-invariant complex structure $J$. Thus, by using \eqref{j}, we obtain that $J$ anti-commutes with $j(z)$ for every $z\in\mz$ and therefore, 
\begin{equation}\label{jJ}
J[j(z),j(z')]=[j(z),j(z')]J, \quad\mbox{ for all } z,z'\in\mz.
\end{equation}

If in addition $(N,g)$ admits non-zero Killing 3-forms, then $j(\mz)$ is a subalgebra, because of Proposition \ref{pro:onedim}.  Hence,  \eqref{jJ} implies that $J$ commutes and anti-commutes with each bracket of elements of $j(\mz)$, whence $j(\mz)$ is an abelian subalgebra of $\so(\mv)$. In particular, using \eqref{j} we obtain
 $$0=[j(z),j(Jz)]=[j(z),Jj(z)] = -2j(z)^2J \quad \mbox{ for all }z\in \mz,$$ which implies $j(z)=0$ for every $z\in\mz$, leading to a contradiction.
\end{proof}

This actually shows that the naturally reductive and complex worlds are disjoint on 2-step nilpotent Lie groups:

\begin{cor}
A naturally reductive $2$-step nilpotent Lie group endowed with a left-invariant metric does not admit orthogonal bi-invariant complex structures. 
\end{cor}

\begin{proof} Indeed, if $(N,g)$ is naturally reductive and has a bi-invariant orthogonal complex structure, each irreducible de Rham factor has the same property, thus carries a non-zero Killing 2-form by Proposition \ref{pro:kb} and a non-zero Killing 3-form by Theorem \ref{th5}, contradicting Proposition \ref{p4}.
\end{proof}

\bibliographystyle{plain}
\bibliography{biblio}

\begin{thebibliography}{10}

\bibitem{AD19}
A.~Andrada and I.~Dotti.
\newblock {Killing-Yano 2-forms on 2-step nilpotent Lie groups}.
\newblock arXiv: 1907.03662.

\bibitem{BDS}
M.L. {Barberis}, I.~{Dotti}, and O.~{Santill\'an}.
\newblock {The Killing-Yano equation on Lie groups.}
\newblock {\em {Classical Quantum Gravity}}, 29(6):1--10, 2012.

\bibitem{lau}
M.L. Barberis, A.~Moroianu, and U.~Semmelmann.
\newblock {Generalized vector cross products and Killing forms on negatively
  curved manifolds.}
\newblock {\em {Geom. Dedicata}}, doi: 10.1007/s10711-019-00467-9., 2019.

\bibitem{bms}
F.~Belgun, A.~Moroianu, and U.~Semmelmann.
\newblock {Killing Forms on Symmetric Spaces.}
\newblock {\em {Differential Geom. Appl.}}, 24:215--222, 2006.

\bibitem{cms}
R.~Cleyton, A.~Moroianu, and U.~Semmelmann.
\newblock {Metric connections with parallel skew-symmetric torsion.}
\newblock arXiv:1807.00191.

\bibitem{dBM}
V.~{del Barco} and A.~{Moroianu}.
\newblock {Symmetric Killing tensors on nilmanifolds}.
\newblock arXiv:1811.09187.

\bibitem{EB}
P.~Eberlein.
\newblock Geometry of {$2$}-step nilpotent groups with a left invariant metric.
\newblock {\em Ann. Sci. \'Ecole Norm. Sup. (4)}, 27(5):611--660, 1994.

\bibitem{gm}
P.~Gauduchon and A.~Moroianu.
\newblock {Killing $2$-forms in dimension $4$}.
\newblock In {\em Special metrics and group actions in geometry}, volume~23 of
  {\em Springer INdAM series}. S.Chiossi, A.Fino, E.Musso, F.Podestà,
  L.Vezzoni (Editors), 2017.

\bibitem{GO2}
C.~Gordon.
\newblock {Naturally reductive homogeneous Riemannian manifolds.}
\newblock {\em {Canad. J. Math.}}, 37:467--487, 1985.

\bibitem{LA}
J.~Lauret.
\newblock Homogeneous nilmanifolds attached to representations of compact {L}ie
  groups.
\newblock {\em Manuscripta Math.}, 99:287--309, 1999.

\bibitem{MA}
L.~Magnin.
\newblock Sur les alg\`ebres de {L}ie nilpotentes de dimension {$\leq 7$}.
\newblock {\em J. Geom. Phys.}, 3(1):119--144, 1986.

\bibitem{a}
A.~Moroianu.
\newblock {Conformally related Riemannian metrics with non-generic holonomy.}
\newblock {\em {J. Reine Angew. Math.}}, 755:279--292, 2019.

\bibitem{au}
A.~Moroianu and U.~Semmelmann.
\newblock {Killing Forms on Quaternion-K\"ahler Manifolds.}
\newblock {\em {Ann. Global Anal. Geom.}}, 28:319--335, 2005.

\bibitem{pw}
R.~Penrose and M.~Walker.
\newblock {On quadratic first integrals of the geodesic equations for type
  $\{22\}$ spacetimes.}
\newblock {\em {Comm. Math. Phys.}}, 18:265--274, 1970.

\bibitem{Se03}
U.~Semmelmann.
\newblock {Conformal Killing forms on Riemannian manifolds}.
\newblock {\em {Math. Z.}}, 245:503--527, 2003.

\bibitem{s}
U.~Semmelmann.
\newblock {Killing forms on ${\rm G}\sb 2$- and ${\rm Spin}\sb 7$-manifolds.}
\newblock {\em {J. Geom. Phys.}}, 56:1752--1766, 2006.

\bibitem{Wi82}
E.~Wilson.
\newblock Isometry groups on homogeneous nilmanifolds.
\newblock {\em Geom. Dedicata}, 12:337--346, 1982.

\end{thebibliography}

\end{document}